\documentclass{birkjour}
\usepackage{amsmath,amscd,amsfonts,amsxtra,latexsym,graphicx,pstricks,amssymb,mathrsfs}
\usepackage{verbatim}
\usepackage{enumerate}

\newtheorem{theorem}{Theorem}[section] % 1st argument is your name for it
\newtheorem{lemma}[theorem]{Lemma}     % 2nd argument is what is printed
\newtheorem{corollary}[theorem]{Corollary}
\newtheorem{proposition}[theorem]{Proposition}
 \theoremstyle{remark}
   \newtheorem{remark}[theorem]{Remark}
   \newtheorem{example}[theorem]{Example}

\renewcommand{\b}{\beta}

\def\la{\langle}
\def\ra{\rangle}

\renewcommand{\phi}{\varphi}

\renewcommand{\emph}{\textsl}
\begin{document}

\title[The Division Algorithm in Sextic Moment Problems]{The Division Algorithm in Sextic \\ Truncated Moment Problems}
\author{Ra\' ul E. Curto}
\address{Department of Mathematics, The University of Iowa, Iowa City, Iowa
52242}
\email{raul-curto@uiowa.edu}
\urladdr{http://www.math.uiowa.edu/\symbol{126}rcurto/}

\author{Seonguk Yoo}
\address{Department of Mathematics, Sungkyunkwan University, Suwon, Korea}
\email{seyoo73@gmail.com}

\begin{abstract}
For a degree $2n$ finite sequence of real numbers $\beta \equiv \beta^{(2n)}= \{ \beta_{00},\beta_{10},$  $\beta_{01},\cdots, \beta_{2n,0}, \beta_{2n-1,1},\cdots, \beta_{1,2n-1},\beta_{0,2n} \}$ to have a representing measure $\mu $, it is necessary for the associated moment matrix $\mathcal{M}(n)$ to be positive semidefinite, and for the
algebraic variety associated to $\beta $, $\mathcal{V}_{\beta} \equiv \mathcal{V}(\mathcal{M}(n))$, to satisfy $\operatorname{rank} \mathcal{M}(n)\leq \operatorname{card} \mathcal{V}_{\beta}$ as well as the following \textit{consistency} condition: if a polynomial $p(x,y)\equiv \sum_{ij}a_{ij}x^{i}y^j$ of degree at most $2n$ vanishes on $\mathcal{V}_{\beta}$, then the \textit{Riesz functional} $\Lambda (p) \equiv p(\beta ):=\sum_{ij}a_{ij}\beta _{ij}=0$.

Positive semidefiniteness, recursiveness,  and the variety condition of a moment matrix are necessary and sufficient conditions to solve the quadratic ($n=1$) and quartic ($n=2$) moment problems. \ Also, positive semidefiniteness, combined with consistency, is a sufficient condition in the case of \textit{extremal} moment problems, i.e., when the rank of the moment matrix (denoted by $r$) and the cardinality of the associated algebraic variety (denoted by $v$) are equal. \ 

For extremal sextic moment problems, verifying consistency amounts to having good representation theorems for sextic polynomials in two variables vanishing on the algebraic variety of the moment sequence. \ We obtain such representation theorems using the Division Algorithm from algebraic geometry. \ As a consequence, we are able to complete the analysis of extremal sextic moment problems. \bigskip
\end{abstract}

\thanks{The first named author was partially supported by NSF Grants
DMS-0801168 and DMS-1302666. \ The second named author was supported by the Brain Korea 21 Program of the National Research Foundation of Korea (S-2016-0032-000).}

\bigskip


\maketitle

\tableofcontents

\setcounter{tocdepth}{2}

%%%%%%%%%%%%%%%%%%%%%%%%%%%%%%%%%%%%%%%%%%%%%%%%%%%%%%%%%%%%%%%%%%%%%%%%%%
%%%%%%%%%%%%%%%%%%%%%%%%%%%%%%%%%%%%%%%%%%%%%%%%%%%%%%%%%%%%%%%
%%%%%%%%%%%%%%%%%%%%%%%%%%%%%%%%%%%%%%%%%%%%%%%%

\section{Introduction}

Given a doubly indexed finite sequence of real numbers $\beta \equiv\beta^{(2n)}= \{\beta_{00},\beta_{10},\beta_{01},\cdots,\beta_{2n,0}$, $\beta_{2n-1,1},\cdots, \beta_{1,2n-1},\beta_{0,2n} \}$,  the truncated real moment problem entails finding a positive Borel measure $\mu$ supported in the real plane $\mathbb{R}^2$ such that
\begin{eqnarray*}
\beta_{ij}=\int x^i y^j \,\, d\mu \,\,\,(i,j\in \mathbb Z_+, \ 0\leq i+j\leq 2n).
\end{eqnarray*}
The sequence $\beta$ is called a \textit{truncated real moment sequence} (of order $2n$) and $\mu$ is referred to as a \textit{representing measure} for $\beta$. \ 
We next associate a moment sequence $\beta$ to the \textit{moment matrix} $\mathcal M(n)(\beta)$ defined by
$$
\mathcal M(n)(\beta):=(\beta_{\mathbf{ i} +\mathbf{j} })_{\mathbf{ i},\, \mathbf{j}\in \mathbb Z^2_+: |\mathbf{i}|, |\mathbf j| \leq n};
$$
we then introduce the following lexicographical order on the rows and columns of $\mathcal{M}(n): \textit{1},X,Y,X^2,X Y,Y^2, \cdots,X^n,\cdots,Y^n$.\ The column space of $\mathcal M(n)$ will be denoted by $\mathcal{C}_{\mathcal M(n)}$.

Let $\mathcal P\equiv \mathbb{R}[x,y]$ be the space of bivariate polynomials and for $k\geq 1$, let $\mathcal P_k$ be the subspace of $\mathcal P$ whose polynomials are of degree less than or equal to $k$.\ The \textit{Riesz functional} $\Lambda\equiv \Lambda_\b:\mathcal P_{2n}\rightarrow \mathbb{R}$ maps a  polynomial $p(x,y)\equiv \sum_{i,j\geq 0;\,0\leq i+j\leq  2n } a_{ij} x^i y^j$ in $\mathcal P_{2n}$ to $\Lambda(p):= \sum_{i,j\geq 0;\,0\leq i+j\leq  2n } a_{ij} \b_{ij}$; it is obvious that the presence of a representing measure $\mu$ implies the identity $\Lambda(p)=\int p\,\, d\mu$. \ We also let $p(X,Y):= \sum_{i,j\geq 0;\,0\leq i+j\leq  2n } a_{ij} X^i Y^j$; $p(X,Y)$ is thus an element in $\mathcal{C}_{\mathcal M(n)}$.

In what follows, $\hat p$ stands for the coefficient vector $(a_{ij})$ of $p$; the fact that $\mathcal M(n)$ is a symmetric matrix can be characterized by the equation
\begin{equation} \label{recureq}
\la \mathcal M(n) \hat p, \hat q\ra = \Lambda (pq) \qquad (p,q\in \mathcal P_n).
\end{equation}

We are now ready to list some necessary conditions for the existence of a representing measure $\mu$ for $\b$. \ We start with positive semidefiniteness of $\mathcal M(n)$; since $\la \mathcal M(n) \hat p, \hat p\ra = \Lambda (p^2)=\int p^2\,\, d\mu$, and $\mathcal M(n)$ is symmetric, it follows that $\mathcal M(n)$ is positive semidefinite (in symbols, $\mathcal M(n)\geq 0$).\

The next necessary condition is related to the support of the measure. \ To date, there is no general algorithm to find the support of a representing measure for the {\it nonsingular} truncated moment problems, while there are concrete results in the singular case. \ Naturally, the atoms of a representing measure lie in the support of the measure; moreover, as it easily follows from Proposition \ref{prop-3-1} below, the support is contained in the \textit{algebraic variety} of $\mathcal M(n)$, which is the intersection of the zero sets of the polynomials arising from column dependence relations in $\mathcal{M}(n)$; that is,
$$
\mathcal V \equiv \mathcal{V}(M(n)) \equiv \mathcal V_\b := \bigcap_{p\in \mathcal P_n,\,\, \hat p \in \ker \mathcal M(n)}\mathcal Z(p),
$$
where $\mathcal Z(p):=\left\{ (x,y)\in \mathbb R^2 : p(x,y)=0\right\}$.

\begin{proposition}\cite[Proposition 3.1]{tcmp1}\label{prop-3-1}
Suppose $\mu$ is a representing measure for $\beta$. \ Then for $p\in \mathcal P_n$,
\begin{equation}
{\operatorname{supp}\;\mu} \subseteq \mathcal Z(p) \Longleftrightarrow p(X,Y)=\mathbf{0}.
\end{equation}
\end{proposition}
\noindent Thus $\operatorname{supp}\;\mu \subseteq \mathcal V$ and, together with  \cite[(1.7)]{tcmp3}, we establish the \textit{variety condition}: if $r:=\operatorname{rank}\;\mathcal M(n)$ and $v:=\operatorname{card}\;\mathcal V$, then
$
r \leq \operatorname{card\;supp}\;\mu \leq v. \
$
Moreover, in  \cite[Corollary 3.4]{tcmp1}, it is shown that the presence of a measure for $\b$ implies 
\begin{equation}\label{RG}
f,g,fg\in \mathcal P_n\text,~ f(X,Y) =\mathbf{0} \Longrightarrow (fg)(X,Y)=\mathbf{0} .
\end{equation}
When (\ref{RG}) holds, $\b$ (or $\mathcal M(n)(\b)$) is said to be \textit{recursively generated}. \ 

Since recursiveness is in general not sufficient for the existence of a representing measure, one considers stronger conditions.\
The first one is called {weak consistency}; $\b$ is said to be \textit{weakly consistent} if the following property holds:
\begin{equation}\label{wcs}
   p\in\mathcal P_n,~p|_{ {\mathcal V}}\equiv 0 \Longrightarrow p(X,Y)=\mathbf{0},
\end{equation}
Indeed, if $\b$ has a representing measure $\mu$, then the moment sequence must be weakly consistent. \ For, if $p\in \mathcal P_n$ and  $p|_{{\mathcal V}}\equiv 0$, then $p|_{\operatorname{supp}\;\mu}\equiv0$ because $\operatorname{supp}\;\mu\subseteq {\mathcal V}$.\
From Proposition \ref{prop-3-1}, we conclude that $p(X,Y)=\mathbf{0}$. \
Similarly, we can show that weakly consistency implies the recursiveness of $\mathcal M(n)(\b)$.\  

Next, we define an even stronger notion that is the key for the main result of this note;  $\b$ is said to be \textit{consistent} if  the following holds:
\begin{equation}\label{cs}
   p\in\mathcal P_{2n},~p|_{ {\mathcal V}}\equiv0 \Longrightarrow \Lambda(p)=0.
\end{equation}
\noindent
It is easy to see that consistency is a necessary condition for the existence of a measure; moreover, \cite[Lemma 2.3]{tcmp11} shows that consistency is fairly strong. \ Consistency alone helps establish the existence of an interpolating linear functional, as a linear combination of point masses (with some densities possibly negative). \ Thus, naturally,  the next question arose in \cite{tcmp11}:

\begin{quote}
{Suppose $\mathcal M(n)(\b)$ is positive and singular. \ If $\b$ satisfies the variety condition and is consistent, does $\b$ admit a representing measure?}
\end{quote}

\medskip
The answer turns out to be negative \cite[Example 3.2]{Fia10}: there exists a consistent real moment matrix $\mathcal M(3)$, with only one column relation ($Y=X^3$), which admits no representing measure. \ As a result, all necessary conditions introduced so far are not sufficient for general truncated moment problems. \ In \cite[Example 3.2]{Fia10} the algebraic variety is infinite; this type of example illustrates that a moment problem with infinite algebraic variety is, generally speaking, more difficult to solve, an aspect one should keep in mind. \ 

In \cite{tcmp11}, R. Curto, L. Fialkow and M. M\"{o}ller initiated the study of extremal moment problems; that is, the case when $r:=\operatorname{rank}\mathcal{M}(n)(\beta )$ and $v:= \operatorname{card} \mathcal{V}_{\beta }$
satisfy $r=v$.\ 
The main result in \cite{tcmp11} follows,
\begin{theorem}\label{thm13} 
For $\beta \equiv \beta ^{(2n)}$ extremal, i.e.,
$r=v$, the following are equivalent:
\begin{enumerate}[(i)]
\item $\beta $ has a representing measure;
\item $\beta $ has a unique representing measure, which is $\operatorname{rank} \mathcal{M}(n)$-atomic;
\item ${\mathcal{M}}(n)$ is positive semidefinite  and $\beta $ is
consistent.
\end{enumerate}
\end{theorem}

For the special case of the column relation $Y=X^{3}$, a detailed study can be found in \cite {tcmp11}; moreover, a complete solution appears in \cite{Fia10}.\
To briefly summarize the main results of this case, observe first that such a bivariate sextic truncated moment problem allows only two extremal cases: $r=v=7$ and $r=v=8$ (see Table 1 in Section \ref{esmp}).\ For the former case, it was proved in \cite[Section 4]{tcmp11} that positivity and recursiveness are sufficient for the existence of representing measures. \ On the other hand, when $r=v=8$, consistency is required.\

%%%%%
Since all the aforementioned conditions do not solve higher order truncated moment problems, a new notion must be discovered; the first one came in \cite{Fia10}, in which a solution of  $\mathcal M(n)(\beta)$, for any $n\geq 1$, satisfying $Y=X^3$ was presented with a numerical condition $\b_{1,2n-1}>\psi(\b)$ (where $\psi(\b)$ is a quantity dependent on the moments), together with positivity and the variety condition.\ This type of numerical condition was never needed in solving quadratic and quartic moment problems, but it was somewhat naturally expected as part of the solution of higher order moment problems.\

Our main results in subsequent sections will confirm that numerical conditions play an important role in solving truncated moment problems. \ 
Our solutions require new cubics or quartics whose zero sets contain the original algebraic variety, together with cubics arising from the column dependence of the moment matrix. \ For these new polynomials, we need to verify consistency; it is this verification that will bring a new type of numerical condition.\  

We briefly pause to remind the reader that there exists an equivalence between bivariate real moments problems and one-variate complex moment problems \cite[Proposition 1.12]{tcmp6}. \ Thus, in what follows we freely switch between real moment matrices $\mathcal{M}(n)$ with columns $\textit{1}, X, Y, X^2, \cdots$ and complex moment matrices $M(n)$ with columns $\textit{1}, Z, \bar{Z}, Z^2, \cdots$.

In this note, our focus is on the bivariate moment problem, with associated ${\mathcal{M}}(3)$ with submatrix ${\mathcal{M}}(2)$ positive and invertible, and with a finite algebraic variety. \ We examine extremal truncated real moment problems with cubic column relations. \ In \cite{CuYoo1}, we discussed a general case of complex sextic moment problems with a cubic  harmonic  column relation, $Z^3=it Z + u\bar Z$ ($t,u\in \mathbb R$). \ In that paper, we showed that the polynomial has zeros symmetrically located, for certain values of $t$ and $u$. \ Indeed, this problem  turns out to \emph{be extremal} and a solution was discovered by checking consistency of the moment matrix after establishing a representation theorem of polynomials vanishing on the algebraic variety $\mathcal V$, namely $\mathcal I_{\mathcal V}$.\

In Sections \ref{sec-r-7} and \ref{sec-r-8}, sextic moment problems with more general cubic column relations in real form will be considered; we will focus on the extremal cases.\ 

For the case of harmonic polynomials of the form $Z^3=it Z + u\bar Z \; (u,t \in \mathbb{R})$, a representation theorem of polynomials vanishing on the algebraic variety (which consists of $7$ points) is achieved by dimensional analysis between a quotient space of complex polynomials and a higher dimensional complex space \cite{CuYoo1}; for more general cases, the Division Algorithm from algebraic geometry will be used instead.\ Conceptually, the set $\mathcal I_{\mathcal V}$ behaves like an ideal and is generated by polynomials from the column relations, plus a few other polynomials obtained from solving Vandermonde equations.\

In this paper we focus on the use of the Division Algorithm as a means of identifying a small collection of polynomials for which the Consistency Property must be verified. \ This explicit listing of just a few polynomials is useful in concrete situations, since it reduces the calculations needed to prove the existence of a representing measure. \ On the other hand, we should also keep in mind the general strategy in \cite{Fia08}, esp. Proposition 3.6 and Remark 3.7, in which L. Fialkow presents a general test for consistency in the case of a finite algebraic variety, and when the points of the variety are known. \ Thus, we wish to make the reader aware that our approach is not necessarily unique, from the point of view of determining solubility from the concrete knowledge of the algebraic variety. \ We emphasize the Division Algorithm, but a different approach could be developed using large Vandermonde matrices and a single matrix-vector equation $A \mathbf{x} = \mathbf{b}$, following the conceptual ideas in \cite{Fia08}.

%%%%%SECTION

\section{The Division Algorithm for Multivariable Polynomials}

In this paper, we settle all remaining cases of the extremal sextic moment problem, via an application of the Division Algorithm in real algebraic geometry.\ The algorithm, which follows, involves polynomials in several indeterminates.

\begin{theorem} \label{divalg} (Division Algorithm \cite{CLO})
Fix a monomial order $>$  on $\mathbb Z^n_{+}$,
and let $F=(f_1, \cdots, f_s)$ be an ordered $s$-tuple of polynomials in  $\mathbb R[x_1,\cdots,x_n]$. \ Then every $f \in\mathbb R[x_1,\cdots,x_n]$ can be written as
$$
f=a_1 f_1 + \cdots + a_s f_s + r,
$$
where $a_i, r  \in \mathbb R[x_1,\cdots,x_n]$, and either $r=0$ or $r$ is a linear combination of monomials with coefficients in $\mathbb R$,
none of which is divisible by any leading terms of $f_1, \cdots, f_s$. \ We call {$r$} a \emph{remainder} of $f$ on division by $F$.\
Furthermore, if $a_i f_i\neq 0,$ then we have
$$
\operatorname{multideg}(f)\geq\operatorname{multideg}(a_i f_i).
$$
\end{theorem}

We shall use the Division Algorithm to prove a structure theorem for polynomials vanishing on the algebraic variety of $\mathcal M(3)$.\ Although not necessary to solve the truncated moment problem, we note that one significant difference between the one-dimensional and multi-dimensional versions of the Division Algorithm lies in the fact that a remainder might not be unique in the latter case.\ Thus, the Division Algorithm  might be regarded as an imperfect generalization of the traditional (one-variable) version. \ However, this situation can be fixed through the so-called \emph{Gr\" {o}ebner basis}. \ For further study on this topic, the reader is referred to \cite[Section 2.3]{CLO}.

\bigskip

\section{Extremal Sextic Moment Problems}\label{esmp}

To date, the most concrete solution of  the truncated moment problem consists of building a sequence of moment matrix extensions which eventually renders a flat moment matrix. \ (A moment matrix $\mathcal{M}(n)$ is said to be {\it flat} if $\operatorname{rank} \mathcal M(n) = \operatorname{rank} \mathcal M(n-1)$.) \  
In principle, one should be able to estimate the smallest integer $k$ satisfying  $\operatorname{rank} \mathcal{M}(n+k)=\operatorname{rank}\;\mathcal{M}(n+k+1).$ \ 
However, for sextic or higher order moment problems, the process does not work very well because it requires too many parameters; this leads to memory overflow when using a software package like {\it Mathematica} \cite{Wol}.\ The existence criterion in \cite[Theorem 1.5]{tcmp3} gives an upper bound of the number of extension matrices; it is $k= 2n^2 + 6n + 6$, generally too large for a practical construction; however, the following theorem made a significant improvement, in the case where the moment sequence has a finite algebraic variety.\
 
\begin{theorem}\cite[Theorem 2.1]{Fia08}\label{min-k}
 Suppose $v<\infty$. \ Then ${\beta}$ admits a representing measure if and only if
  $\mathcal M (n)({\beta})$ has a positive extension $\mathcal M (n+v-r+1)$ satisfying $\operatorname{rank} \mathcal M (n+v-r+1)\leq\operatorname{card}\;\mathcal V(\mathcal M (n+v-r+1)).$
\end{theorem}

\begin{corollary} \label{newcor}
Let $\beta$, $n$, $v$ and $r$ be as in Theorem \ref{min-k}, and assume that $v=r$. \ Then $n+v-r+1=n+1$, so that there exists a representing measure if and only if $\mathcal{M}(n)(\beta)$ has a positive extension $\mathcal{M}(n+1)$ satisfying $\operatorname{rank} \mathcal{M}(n+1) \le \operatorname{card} \; \mathcal{V} (\mathcal{M}(n+1))$.
\end{corollary}

When Corollary \ref{newcor} is applied to the extremal case ($r=v$) of the sextic moment problem ($n=3$), it implies that the existence of a representing measure is equivalent to the  existence of a moment matrix extension $\mathcal{M}(4)$ such that $\operatorname{rank} \mathcal{M}(4) \le \operatorname{card} \; \mathcal{V}(\mathcal{M}(4))$. \ 

Theorem \ref{min-k} states that if a moment sequence admits a representing measure, $\mathcal M (n)$ must have an ascending sequence of moment matrix extensions (rank-increasing) with length $k= v-r$. \ 
As a result, we can classify all sextic truncated moment problems, as follows: let us consider an ascending sequence of moment matrix extensions
$$
\mathcal M (n)\rightarrow \mathcal M (n+1)\rightarrow \mathcal M (n+2)\rightarrow \cdots.
$$
Now denote $r_n:=\operatorname{rank}\;\mathcal M (n)$ and $v_n:=\operatorname{card}\;\mathcal  V(\mathcal M (n))$.\ 
The variety condition must hold, so we obtain the following chain of inequalities:
$$
{r_n\leq r_{n+1}\leq r_{n+2} \leq \cdots \leq v_{n+2} \leq v_{n+1}\leq v_n}.
$$
Since truncated moment problems with quadratic column relations are well understood, we may assume, without loss of generality, that our moment matrix $\mathcal M (3)$ has an invertible $\mathcal M (2)$ block; that is, we will always assume that $\mathcal M (2)$ is positive definite.\  Table 1 depicts a classification of sextic moment problems, in terms of the rank of the moment matrix and the cardinality of its algebraic variety. \
For the cases with a finite algebraic variety, it shows the maximum length of the extension sequence (see the column labeled Max Extension).

\begin{table}[hptp]\label{table1}
\centering
\begin{tabular}{|c|c|c|c|c|c|c|}
  \hline
  $r_3$ & $v_3$ & $v_3-r_3$ & Max Extension& & Ex w rep meas & Ex w/o rep meas\\
  \hline
  \hline
  7 &  7 &  0 & $\mathcal M (4)$ & extremal & known & unknown\\
 \hline
   7 & 8 & 1 & $\mathcal M (5)$ & & unknown & unknown\\
 \hline

   7 & 9 & 2 & $\mathcal M (6)$ & & unknown & unknown\\
  \hline

  7 & $\infty$ &  N/A & N/A & & known & known\\
  \hline

  8 &  8 & 0 & $\mathcal M (4)$ & extremal& known & known\\
\hline
   8 & 9 & 1 & $\mathcal M (5)$ &  & known & known\\
 \hline

 8 & $\infty$ &  N/A & N/A & & known & known\\
  \hline

 9 & 9 &0  &N/A  &  & impossible & impossible\\
\hline

9 & $\infty$ &  N/A & N/A & & known & known\\
  \hline

10 & $\infty$ &  N/A & N/A & & known & known\\
  \hline
\end{tabular}
\vskip .2cm
\caption{Classification of sextic moment problems in terms of $r$ and $v$}
\end{table}

As noted in Table 1, the case of $r_3=v_3=9$ cannot happen, since the presence of exactly one column relation means that the associated algebraic curve is the algebraic variety, and this immediately means that $v_3$ is infinity. \ 

In this note,  we are mostly interested in the extremal cases, but we will also see how some non-extremal cases of $\mathcal{M}(n)$ can be treated as extremal moment problems for an extension $\mathcal{M}(n+1)$, after allowing proper new moments. \ For this, we need two more results from \cite{tcmp10}. 

\begin{theorem}\cite[Theorem 2.3]{tcmp10}\label{fe-rv}
If $\mathcal M(n)\geq 0$ admits a flat extension $\mathcal M(n+1)$ (i.e., $\operatorname{rank} \mathcal M(n+1) = \operatorname{rank} \mathcal M(n)$),
then $\operatorname{rank} \mathcal M(n)=\operatorname{card} \mathcal V(\mathcal M(n+1))$ and $\mathcal V(\mathcal M(n+1))$ forms
the support of the unique representing measure $\nu$ for $\mathcal M(n+1)$.
\end{theorem}

\begin{theorem}\cite[Theorem 2.4]{tcmp10} \label{fe-rv2}
Assume that $\mathcal M(n)\geq 0$ admits a flat extension $\mathcal M(n+1)$, and let $\mathcal{M}(n+2)$ the subsequent flat extension of $\mathcal{M}(n+1)$. \ Then $\mathcal V(\mathcal M(n+2))=\mathcal V(\mathcal M(n+1)).$
\end{theorem}

\noindent Combining Theorems \ref{fe-rv} and \ref{fe-rv2}, we find a necessary condition for the existence of a flat extension:
\begin{equation}\label{rv-seq}
 r_n=r_{n+1} \Longrightarrow   \left\{
                                   \begin{array}{rcl}
                                     r_n&=&v_{n+1}\\
                                     v_{n+1} &=& v_{n+2}
                                   \end{array}
                                 \right\} \Longrightarrow r_n= v_{n+1} = v_{n+2}.
\end{equation}
It follows that $r_n=r_{n+1} \Longrightarrow r_{n+1}=v_{n+1}$, and therefore $\mathcal{M}(n+1)$ must be extremal. \ 

If we now focus on the case $\mathcal M(3)\geq 0$ with $r_3=8$ and $v_3=9$ and apply (\ref{rv-seq}), we observe that there exist two feasible cases:
\begin{table}[hptp]
\centering
\begin{tabular}{|ccccccccccc|}
\hline
  $r_3$ & $\leq$ & $r_4$ & $\leq$ & $r_5 $&$ \leq $&$ v_5 $&$ \leq $&$ v_4 $&$\leq $&$v_3$ \\
\hline
8 & &$\boldsymbol{8}$ & & 8 & & 8 & &$\boldsymbol{8}$ & & 9\\
\hline
8 & & $\boldsymbol{9}$ & & 9 & & 9 & & $\boldsymbol{9}$& & 9  \\
\hline
\end{tabular}
\vskip .2cm
\caption{The cases when $r_3=8$ and $v_3=9$}
\end{table}

\noindent In both cases $\mathcal{M}(4)$ must be extremal. \ Consequently, after building $\mathcal M(4)$ with proper higher order moments, we may resolve the cases by checking consistency. \ We will take a careful look at extremal extensions $\mathcal{M}(4)$ of $\mathcal{M}(3)$ in future work. \ In Sections \ref{sec-r-7} and \ref{sec-r-8} below, we dispose of the remaining extremal cases of the sextic moment problem.

\bigskip

%%%%%%%%%%%%%%%%%%%%%%%%%%%%%%%%%%
%%%%%%%%%%%%%%%%%%%%%%%%%%%%%%%%%%

\section{The Case ${\operatorname{rank}\;\mathcal{M} (3)=\operatorname{card}\; \mathcal V=7}$ }\label{sec-r-7}
Let $\{\emph{1}, X, Y, X^2, X Y, Y^2,X^3, X^2 Y, XY^2,Y^3\}$ denote the set of columns in $\mathcal M(3)$. \ In the sequel, assume $\operatorname{rank}\mathcal M (3):=r$ and $\operatorname{card}\;\mathcal V_{\b}:=v$, and write $\mathcal V_{\b} \equiv \{(x_1,y_1),\cdots, (x_v,y_v)\}$. \ The generalized Vandermonde matrix associated with $\mathcal{M}(3)$ is 
\begin{equation}\label{}
W:=\begin{pmatrix}
1&x_1&y_1&{x_1}^2& x_1 y_1&{y_1}^2 &{x_1}^3 & x_1^2 y_1& x_1y_1^2 & {y_1}^3\\
1&x_2&y_2&{x_2}^2& x_2 y_2&{y_2}^2 &{x_2}^3 & x_2^2 y_2& x_2y_2^2 & {y_2}^3\\
\vdots & \vdots & \vdots & \vdots & \vdots & \vdots & \vdots & \vdots & \vdots & \vdots \\
1&x_v&y_v&{x_v}^2& x_v y_v&{y_v}^2 &{x_v}^3 & x_v^2 y_v& x_vy_v^2 & {y_v}^3
\end{pmatrix}.
\end{equation}

Just as we do for the columns of $\mathcal{M}(3)$, we will label the columns of $W$ using $\textit{1}, X, Y, \cdots$. \ Given a basis $\mathcal{B}$ for the column space of $\mathcal{M}(3)$, let $W_{\mathcal B}$ denote the compression of the generalized Vandermonde matrix $W$ to $\mathcal B$. \  

In this section we study the extremal moment problem for a moment matrix $\mathcal M(3)$ satisfying
\begin{equation}\label{rv7}
\mathcal M (3) \geq 0, \mathcal M (2)>0, \text{ and }r=v=7.
\end{equation}
Set $\mathcal V\equiv \mathcal V_{\beta}:=\{(x_1,y_1),\cdots, (x_7,y_7)\}$. \ Because of the variety condition and the invertibility of $\mathcal M (2)$, there is only one linearly independent column amongst $X^3$, $X^2 Y$, $XY^2$, and $Y^3$. \ Thus, the monomial basis of the column space $\mathcal C_{\mathcal M (3)}$ must be one of the following:
\medskip
\begin{itemize}
  \item []Case 1. \ $\mathcal B_1:=\{\textit{1},X,Y,X^2,XY,Y^2,X^3\}$
\medskip
  \item[] Case 2. \ $\mathcal B_2:=\{\textit{1},X,Y,X^2,XY,Y^2,X^2 Y\}$
\medskip
\item[] Case 3. \ $\mathcal B_3:=\{\textit{1},X,Y,X^2,XY,Y^2,X Y^2\}$
\medskip
 \item[] Case 4. \ $\mathcal B_4:=\{\textit{1},X,Y,X^2,XY,Y^2,Y^3\}$
\end{itemize}

\medskip
Recall that consistency of $\beta$ is a necessary condition for the existence of a representing measure, and it is also sufficient in the extremal case. \ Thus, the key to the solution of the truncated moment problem is checking consistency of the moment sequence. \ Weak consistency and a numerical condition about moments solve Case 1, as follows. \ But first, we recall a useful property of weak consistency.

\begin{lemma}
\label{lem31}(\cite[Lemma 2.7]{tcmp11}) \ The following are equivalent for $\beta $ extremal: \newline
i) $\beta$ is weakly consistent; \newline
ii) For any basis ${\mathcal{B}}$ of $\mathcal{C}_{\mathcal{M}(n)}$, the Vandermonde matrix $W_{\mathcal{B}}$ is
invertible; \newline
iii) There exists a basis ${\mathcal{B}}$ of $\mathcal{C}_{\mathcal{M}(n)}$
such that the Vandermonde matrix $W_{\mathcal{B}}$ is invertible.
\end{lemma}
 
\begin{theorem} \ (Case 1)
\ Suppose $\mathcal M (3)(\beta)$ satisfies (\ref{rv7}). \ Let $\mathcal B_1$
%$\mathcal B_1:=\{1,X,Y,X^2,XY,Y^2,X^3\}$
be a basis for $\mathcal C_{\mathcal M (3)}$.
Then $\beta$ has a representing measure if and only if
$\mathcal M (3)$ is {weakly consistent} and
for $0\leq i+j\leq2$,
$$
\Lambda_\b(x^i y^j(x^4-a_{00} - a_{10} x-a_{01} y - a_{20} x^2- a_{11} x y - a_{02} y^2 - a_{30} x^3))=0,
$$
where $(a_{00} , a_{10} ,a_{01}, a_{20} , a_{11}, a_{02} , a_{30} )^T=W_{\mathcal B_1}^{-1} (x_1^4, \cdots, x_7^4)^T$.
\end{theorem}

\begin{proof} Let $q_k(X,Y)=\mathbf{0}$ denote the column relation in the $k$-th column of $\mathcal{M}(3)$ for $k=8,9,10$.
\ Since 
\begin{equation}\label{}
W_{\mathcal B_1} \equiv \begin{pmatrix}
1&x_1&y_1&{x_1}^2& x_1 y_1&{y_1}^2 &{x_1}^3\\
1&x_2&y_2&{x_2}^2&x_2 y_2 &{y_2}^2 &{x_2}^3\\
\vdots & \vdots & \vdots & \vdots & \vdots & \vdots & \vdots \\
1&x_7&y_7&{x_7}^2&x_7 y_7 &{y_7}^2 & {x_7}^3
\end{pmatrix}
\end{equation}
is invertible, there exists a unique polynomial with the leading monomial $x^4$ that vanishes on the variety $\mathcal V$, say,
\begin{eqnarray*}
s(x,y):= x^4-(a_{00} + a_{10} x+a_{01} y + a_{20} x^2+ a_{11} x y + a_{02} y^2 + a_{30} x^3),
\end{eqnarray*}
where $(a_{00} , a_{10} ,a_{01}, a_{20} , a_{11}, a_{02} , a_{30} )^T=W_{\mathcal B_1}^{-1} (x_1^4, \cdots, x_7^4)^T$.

Set $\mathcal I:=\{p\in \mathcal P_6: p|_{\mathcal V}\equiv 0\}$. \ 
Applying the Division Algorithm (Theorem \ref{divalg}), any $p\in \mathcal I$ can be written as
\begin{eqnarray*}
p=Aq_8+B q_{9} +C q_{10} + Ds+r,
\end{eqnarray*}
where $A,B,C\in\mathcal P_3$, $D\in\mathcal P_2$ and
$r(x,y)= c_{00} + c_{10} x+c_{01} y + c_{20} x^2+ c_{11} x y + c_{02} y^2 + c_{30} x^3$ for some $ c_{00}, \cdots ,c_{02}, c_{30}\in\mathbb R$.

We now claim that $\mathcal I = \{ f q_8+ g q_9+h q_{10} + qs:f,g,h\in \mathcal P_3, q\in \mathcal P_2\} $ by showing $r(x,y)\equiv 0$. \ Note that since $p$ vanishes on $\mathcal V$, so does $r$, which leads to the following linear system:
$$
W_{\mathcal B_1}
\left(\begin{array}{ccccccc}
c_{00} &c_{10} &c_{01} &c_{20}&c_{11}&c_{02}&c_{30}
\end{array} \right)^T
= \left(\begin{array}{ccc}
0&\cdots&0
\end{array}\right)^T.
$$
The matrix in the left-hand side is invertible, so we know $c_{00} =c_{10} =c_{01} =c_{20} =c_{11}=c_{02}$ \linebreak $ =c_{30}=0$, which means $r(x,y) \equiv 0$.

Consequently, $\beta$ is consistent if and only if
\begin{eqnarray*}
\left\{
\begin{array}{ll}
 \Lambda_{\beta}(x^i y^j q_k(x,y))=0 & (0\leq i+j\leq 3 ;\, k=8,9,10); \\
 \Lambda_{\beta}(x^t y^u s(x,y))=0 & (0\leq t+u\leq 2).
                \end{array}\right.
\end{eqnarray*}
But both conditions are immediate from (\ref{recureq}), the column relations in $\mathcal M (3)$, and the hypotheses.
\end{proof}

For Case 2, weak consistency, together with positivity and the variety condition, is enough to solve the moment problem; the proof mimics that of Case 1.

\begin{theorem} \ (Case 2)
\ Suppose $\mathcal M (3)(\beta)$ satisfies (\ref{rv7}). \ Let $\mathcal B_2$
be a basis for $\mathcal C_{\mathcal M (3)}$.
Then $\beta$ admits a representing measure if and only if
$\mathcal M (3)$ is {weakly consistent}.
\end{theorem}

\begin{proof}
Let $q_k(X,Y)=\mathbf{0}$ is the column relation in $i$th column for $k=7,9,10$. \ 
Set $\mathcal I:=\{p\in \mathcal P_6: p|_{\mathcal V}\equiv 0\}$. \ 
Due to the Division Algorithm, we may write any $p\in \mathcal I$ as
\begin{eqnarray*}
p=Aq_7+B q_{9} +C q_{10} + r,
\end{eqnarray*}
where $A,B,C\in\mathcal P_3$ and
$r(x,y)= c_{00} + c_{10} x+c_{01} y + c_{20} x^2+ c_{11} x y + c_{02} y^2 + c_{21} x^2 y$ for some $ c_{00}, \cdots ,c_{02}, c_{21}\in\mathbb R$.

Once we establish that $r(x,y)\equiv 0$, we may verify that \linebreak
$\mathcal I = \{ f q_7+ g q_9+h q_{10}:f,g,h\in \mathcal P_3\} $. \ Note that 
$p\vert_{\mathcal V}\equiv 0 \Longrightarrow r\vert_{\mathcal V}\equiv 0$. \ 
This argument brings up the following linear system:
\begin{eqnarray*}
W_{\mathcal B_2}
\begin{pmatrix}
c_{00} &c_{10} &c_{01} &c_{20}&c_{11}&c_{02}&c_{21}
\end{pmatrix}^T =
\begin{pmatrix}
0&\cdots&0
\end{pmatrix}^T.
\end{eqnarray*}
The invertibility of $W_{\mathcal B_2}$ implies that $c_{00} =c_{10} =c_{01} =c_{20} =c_{11} =c_{02} =c_{21}=0$, \linebreak
i.e., $r(x,y) \equiv 0$. \ Therefore, we need only three polynomials attained from column relations in $\mathcal M (3)$ to check consistency of $\beta$. \ The test is straightforward, since $\beta$ is consistent if and only if
\begin{eqnarray*}
 \Lambda_{\beta}(x^i y^j q_k(x,y))=0 & (0\leq i+j\leq 3 ;\, k=7,9,10),
\end{eqnarray*}
which is inherent in $\mathcal M (3)$, using (\ref{recureq}). \ This completes the proof.
\end{proof}

In order to solve Case 3, we can use the same approach as in Case 2; the proof is omitted.

\begin{theorem} \ (Case 3) \
Suppose $\mathcal M (3)(\beta)$ satisfies (\ref{rv7}). \ Let $\mathcal B_3$ %$:=\{1,X,Y,X^2,XY,Y^2,X Y^2\}$
be a basis for $\mathcal C_{\mathcal M (3)}$.
Then $\beta$ admits a representing measure if and only if
$\mathcal M (3)$ is {weakly consistent}.
\end{theorem}

\begin{remark} \ (Case 4) \
It is straightforward to see that Case 4 reduces to Case 1 via the invariance of moment problems under degree-one transformations (\cite[Proposition 1.7]{tcmp6}). 
\ Indeed, it suffices to consider the degree-one transformation that interchanges $X$ and $Y$. $\hfill\square$
\end{remark}

\bigskip

\section{The Case ${\operatorname{rank}\;\mathcal M (3)=\operatorname{card}\; \mathcal V=8}$ }\label{sec-r-8}
In this section we discuss the other extremal moment problem for a moment matrix $\mathcal M (3)$ satisfying
\begin{equation}\label{rv8}
\mathcal M (3) \geq 0, \mathcal M (2)>0, \text{ and }r=v=8.
\end{equation}
Write $\mathcal V\equiv \mathcal V_{\beta}:=\{(x_1,y_1),\cdots, (x_8,y_8)\}$. \ Since we assumed the invertibility of the minor block $\mathcal M (2)$ and $\operatorname{rank} \mathcal M (3)=8$, there are two linearly independent columns among $X^3, X^2 Y, XY^2$, and $Y^3$. \ Thus, there are six natural choices for bases of $\mathcal C_{\mathcal M (3)}$:
\medskip
\begin{itemize}
  \item []Case 1. \ $\mathfrak B_1:=\{\textit{1},X,Y,X^2,XY,Y^2,X^3,X^2 Y\}$
\medskip
   \item []Case 2. \ $\mathfrak B_2:=\{\textit{1},X,Y,X^2,XY,Y^2,X^3,X Y^2\}$
\medskip
 \item []Case 3. \ $\mathfrak B_3:=\{\textit{1},X,Y,X^2,XY,Y^2,X^3, Y^3\}$
\medskip
  \item[] Case 4. \ $\mathfrak B_4:=\{\textit{1},X,Y,X^2,XY,Y^2,X^2 Y, XY^2\}$
\medskip
\item[] Case 5. \ $\mathfrak B_5:=\{\textit{1},X,Y,X^2,XY,Y^2,  X^2 Y, Y^3\}$
\medskip
\item[] Case 6. \ $\mathfrak B_6:=\{\textit{1},X,Y,X^2,XY,Y^2,X Y^2,Y^3\}$

\end{itemize}

\medskip
Again, the problem will be solved once we show that the moment sequence is consistent. \ The proofs are very similar to those in the previous section, but we need to find another polynomial vanishing on the algebraic variety. \ Also, we can reduce some cases to a subcase of another via degree-one transformations, that is, interchanging $X$ and $Y$. \ We immediately obtain the following two Claims.
\begin{flushleft}
\medskip
Claim 1. \ Case 6 is a subcase of Case 1.
\end{flushleft}

\begin{flushleft}
\medskip
Claim 2. \ Case 5 is a subcase of Case 2.
\end{flushleft}

\medskip
Thus, it suffices to focus attention on the first four cases and present their solutions. \ The detailed proofs are omitted. \ Instead, we present in the next paragraph a general sketch of a typical proof; with this, the interested reader will be able to fill out the remaining details. \ 

Obviously, the ``only if'' parts of the proofs of Theorems \ref{thm51}, \ref{thm52}, \ref{thm53} and \ref{thm-rv8-4} are trivial, so we work on the converses. \ To establish consistency, we need to show that the Riesz functional is zero for any polynomial vanishing on the algebraic variety, of degree less than or equal to $6$. \ Thus, it is essential to construct a representation of such polynomials, which is done by the Division Algorithm. \ The representing sets contain at most 4 polynomials, two of which come from column relations in $\mathcal M(3)$ and the other two polynomials (which are quartic) are found using the invertibility of the appropriate compression of the generalized Vandermonde matrix (equivalently, by the weak consistency of $\beta$). \ Multiplying polynomials in the representing set by suitable monomials, we can check that the Riesz functional is zero for higher order polynomials.

\begin{theorem} \label{thm51} \ (Case 1) \
Suppose $\mathcal M (3)(\beta)$ satisfies (\ref{rv8}). \ Let $\mathfrak B_1$
be a basis for $\mathcal C_{\mathcal M (3)}$.
Then $\beta$ has a representing measure if and only if
$\mathcal M (3)$ is {weakly consistent} and
for $0\leq i+j\leq2$,
$$\Lambda_\b(x^i y^j(x^4-a_{0} - a_{1} x-a_{2} y - a_{3} x^2- a_{4} x y - a_{5} y^2 - a_{6} x^3 - a_{7} x^2 y ))=0$$
and
$$\Lambda_\b(x^i y^j(x^3y-b_{0} - b_{1} x-b_{2} y - b_{3} x^2- b_{4} x y - b_{5} y^2 - b_{6} x^3- b_{7} x^2y ))=0,$$
where $(a_{0} , a_{1} ,a_{2}, a_{3} , a_{4}, a_{5} , a_{6},a_{7} )^T=W_{\mathfrak B_1}^{-1} (x_1^4, \cdots, x_8^4)^T$\\
and $(b_{0} , b_{1} ,b_{2}, b_{3} , b_{4}, b_{5} , b_{6},b_{7})^T=W_{\mathfrak B_1}^{-1} (x_1^2 y_1, \cdots, x_8^2 y_8)^T$.
\end{theorem}

\begin{theorem} \label{thm52}\ (Case 2) \
Suppose $\mathcal M (3)(\beta)$ satisfies (\ref{rv8}). \ Let $\mathfrak B_2$
%$:=\{1,X,Y,X^2,XY,Y^2,X^3,XY^2\}$
be a basis for $\mathcal C_{\mathcal M (3)}$. \ 
Then $\beta$ has a representing measure if and only if
$\mathcal M (3)$ is {weakly consistent} and
for $0\leq i+j\leq2$,
$$\Lambda_\b(x^i y^j(x^4-a_{0} - a_{1} x-a_{2} y - a_{3} x^2- a_{4} x y - a_{5} y^2 - a_{6} x^3 - a_{7} x^2 y ))=0,$$
where $(a_{0} , a_{1} ,a_{2}, a_{3} , a_{4}, a_{5} , a_{6},a_{7} )^T=W_{\mathfrak B_2}^{-1} (x_1^4, \cdots, x_8^4)^T$.
\end{theorem}

\begin{theorem} \label{thm53}\ (Case 3) \ Suppose $\mathcal M (3)(\beta)$ satisfies (\ref{rv8}). \ Let $\mathfrak B_3$
%$:=\{1,X,Y,X^2,XY,Y^2,X^3, Y^3\}$
be a basis for $\mathcal C_{\mathcal M (3)}$. \ 
Then $\beta$ has a representing measure if and only if
$\mathcal M (3)$ is {weakly consistent} and
for $0\leq i+j\leq2$,
$$\Lambda_\b(x^i y^j(x^4-a_{0} - a_{1} x-a_{2} y - a_{3} x^2- a_{4} x y - a_{5} y^2 - a_{6} x^3 - a_{7} x^2 y ))=0$$
and
$$\Lambda_\b(x^i y^j(y^4-b_{0} - b_{1} x-b_{2} y - b_{3} x^2- b_{4} x y - b_{5} y^2 - b_{6} x^3- b_{7} x^2y ))=0,$$
where $(a_{0} , a_{1} ,a_{2}, a_{3} , a_{4}, a_{5} , a_{6},a_{7} )^T=W_{\mathfrak B_3}^{-1} (x_1^4, \cdots, x_8^4)^T$\\
and $(b_{0} , b_{1} ,b_{2}, b_{3} , b_{4}, b_{5} , b_{6},b_{7})^T=W_{\mathfrak B_3}^{-1} ( y_1^4, \cdots,y_8^4)^T$.
\end{theorem}

\begin{theorem} \label{thm-rv8-4} \ (Case 4) Suppose $\mathcal M (3)(\beta)$ satisfies (\ref{rv8}). \ Let $\mathfrak B_4$
%$:=\{1,X,Y,X^2,XY,Y^2,X^3,X^2 Y\}$
be a basis for $\mathcal C_{\mathcal M (3)}$. \ 
Then $\beta$ has a representing measure if and only if
$\mathcal M (3)$ is {weakly consistent} and
for $0\leq i+j\leq2$,
$$\Lambda_\b(x^i y^j(x^2 y^2-a_{0} - a_{1} x-a_{2} y - a_{3} x^2- a_{4} x y - a_{5} y^2 - a_{6} x^3 - a_{7} x^2 y ))=0,$$
where $(a_{0} , a_{1} ,a_{2}, a_{3} , a_{4}, a_{5} , a_{6},a_{7} )^T=W_{\mathfrak B_4}^{-1} (x_1^2 y_1^2, \cdots, x_8^2 y_8^2)^T$.
\end{theorem}

(Observe that in Theorem \ref{thm-rv8-4}, we are essentially representing the column $X^2Y^2$ in terms of columns of lower degrees.) \ In \cite{tcmp11}, the authors provided an example of $\mathcal M (3)$ satisfying $Y=X^3$ in $\mathcal C_{\mathcal M (3)}$ and $r=v=8$, which is weakly consistent but not consistent. \ Consequently, $\b^{(6)}$ has no representing measure. \ They actually obtained a rather general class of examples; here, we will content ourselves with focusing on one specific case, to show that the numerical conditions about moments in the our main theorems are essential.

\begin{example} \label{ex16} (cf. \cite[Theorem 5.2]{tcmp11}) \ Consider $\b^{(6)}$ with following moments;
\begin{eqnarray*}
  \begin{array}{llll}
    \b_{00}:=14 ,& \b_{10}:=\frac{7}{2} ,& \b_{01}:=-\frac{67}{8} ,& \\

\medskip \b_{20}:=\frac{79}{4} ,& \b_{11}:= \frac{1055}{16} ,& \b_{02}:= \frac{18195}{64},&\\

\medskip    \b_{30}:=-\frac{67}{8} ,& \b_{21}:= -\frac{1935}{32} ,& \b_{12}:=-\frac{43115}{128} ,& \b_{03}:=-\frac{926695}{512}, \\

\medskip    \b_{40}:= \frac{1055}{16},& \b_{31}:= \frac{18195}{64},& \b_{22}:= \frac{336151}{256},& \\

\medskip    \b_{13}:= \frac{6407195}{1024},& \b_{04}:= \frac{124731423}{4096},& &\\

\medskip    \b_{50}:=-\frac{1935}{32} ,& \b_{41}:=-\frac{43115}{128}  ,& \b_{32}:=-\frac{926695}{512} ,& \\
\medskip    \b_{23}:= -\frac{19736547}{2048} ,& \b_{14}:=-\frac{419176415}{8192} ,& \b_{05}:=-\frac{8894873563}{32768}  ,& \\
\medskip    \b_{60}:=\frac{18195}{64} ,& \b_{51}:=\frac{336151}{256}  ,& \b_{42}:=\frac{6407195}{1024} ,& \b_{33}:=\frac{124731423}{4096},  \\
\medskip     \b_{24}:=\frac{2469281827}{16384} ,& \b_{15}:=\frac{49568350247}{65536} ,&\b_{06}:= \frac{1006568996907}{262144}. &\\
  \end{array}
\end{eqnarray*}
%Indeed, the moment are generated by the equation, $\b_{ij}\equiv \b(i,j):= (i+3j)2^{1-i-3j} + \sum_{k=1}^7 x_k^i y_k^j  + 7x_8^i y_8^j.$
After building $\mathcal M (3)(\beta)$, we see that there are two column relations $$f(X,Y):=X^3-Y=\mathbf{0}$$ and $$g(X,Y):=Y^3 -3X + \frac{3}{4}Y +13X^2 -\frac{65}{4}XY + \frac{13}{4}Y^2 -12X^3 +22 X^2 Y -\frac{35}{4}XY^2=\mathbf{0}$$ in $\mathcal C_{\mathcal M (3)}$. \ This moment matrix satisfies (\ref{rv8}) with the basis $\mathfrak B_4$ as in Case 4. \ A calculation shows that the algebraic variety is 
\begin{eqnarray*}
\mathcal V & = & \{(0,0),(-2,-8),(2,8),(1,1), \left((-1-\sqrt{13})/2,-5-2\sqrt{13}\right), \\
& & \left((-1+\sqrt{13})/2,-5+2\sqrt{13}\right), (-1,-1),\left(1/2,1/8 \right)\}
\end{eqnarray*}
(see \cite[Section 6]{tcmp11}). \ In order to apply Theorem \ref{thm-rv8-4}, we need to find the new polynomial, denoted as $h(x,y)$, vanishing on $\mathcal V$ by using the compression of the generalized Vandermonde matrix. \ Theorem 6.2 in \cite{tcmp11} states that the above mentioned moment matrix $\mathcal{M}(3)$ is positive semidefinite, recursively generated, and extremal, but does not admit a representing measure (because it is not consistent). \ Using our Theorem \ref{thm-rv8-4}, we can proceed as follows. \ First, 
\begin{equation*}
 h(x,y)=x^4+ 6 x-\frac{11}{2} y-14x^2+\frac{43 }{2}x y-\frac{17}{2} y^2 -x^2 y +\frac{1}{2}x y^2.
\end{equation*}
Next, we evaluate the Riesz functional acting on $h$:
\begin{eqnarray*}
    \Lambda_\b (h)&=& \b_{40}+ 6 \b_{10}-\frac{11}{2}\b_{01}-14\b_{20}+\frac{43 }{2}\b_{11}-\frac{17}{2}\b_{02} -\b_{21} +\frac{1}{2}\b_{12}\\
    &=&-\frac{320081}{256},
\end{eqnarray*}
which is different from zero. \ That is, even though $h$ vanishes on $\mathcal V$, its Riesz functional is not zero. \ We have therefore verified that $\b$ admits no representing measure. $\hfill\square$

\end{example}

\bigskip
\textit{Acknowledgment}. \ Example \ref{ex16}, and portions of the proofs of
some results in this paper were obtained using calculations with the
software tool \textit{Mathematica \cite{Wol}}.

\bigskip

Department of Mathematics, The University of Iowa, Iowa City, Iowa 52242; \ raul-curto@uiowa.edu

\medskip
Department of Mathematics, Sungkyunkwan University, Suwon, Korea; \ seyoo73@gmail.com


\begin{thebibliography}{99}

\bibitem{CLO} D. Cox, J. Little and D. O'Shea, \textit{Ideals,
Varieties and Algorithms: An Introduction to Computational Algebraic
Geometry and Commutative Algebra}, Second Edition, Springer-Verlag, New
York, 1992.

\bibitem{tcmp1} R. Curto and L. Fialkow, Solution of the truncated
complex moment problem with flat data, \textit{Memoirs Amer. Math. Soc} 119, no.
568, Amer. Math. Soc., Providence, 1996; x+52 pp.

\bibitem{tcmp3} R. Curto and L. Fialkow, Flat extensions of positive
moment matrices: Recursively generated relations, \textit{Memoirs Amer.
Math. Soc} 136, no. 648, Amer. Math. Soc., Providence, 1998; x+56 pp.

\bibitem{tcmp6} R. Curto and L. Fialkow, Solution of the singular
quartic moment problem, J. Operator Theory 48(2002), 315--354.

\bibitem{tcmp10} R. Curto and L. Fialkow, Truncated $K$-moment
problems in several variables, J. Operator Theory 54(2005), 189--226.

\bibitem{tcmp11} R. Curto, L. Fialkow and H.M. M\"{o}ller, The
extremal truncated moment problem, Integral Equations Operator Theory 60(2008), 177--200.

\bibitem{CuYoo1} R. Curto and S. Yoo, Cubic column relations in truncated moment problems, J. Funct. Anal. 266 (2014), 1611--1626.

\bibitem{Fia08} L. Fialkow, Truncated multivariable moment problems
with finite variety, J. Operator Theory 60(2008), 343--377.

\bibitem{Fia10} L. Fialkow, Solution of the truncated moment problem with variety $y=x^3$, Trans. Amer. Math. Soc. 363(2011), 3133--3165.

\bibitem{Wol} Wolfram Research, Inc., \textit{Mathematica}, Version 9.0, Champaign, IL, 2013.

\bibitem{Yoo} S. Yoo, \textit{Extremal Sextic Truncated Moment Problems}, Ph.D. Dissertation, Univ. of Iowa, Iowa City, USA, 2011.

\end{thebibliography}
\end{document}